\documentclass[12pt]{amsart}
\input{amssym.def}

\newtheorem{theorem}{Theorem}
\newtheorem{lemma}{Lemma}
\newtheorem{klemma}[lemma]{$(k)$ Lemma}

\newenvironment{definition}[1][Definition]{\begin{trivlist}
\item[\hskip \labelsep {\bfseries #1}]}{\end{trivlist}}
\newenvironment{remark}[1][Remark]{\begin{trivlist}
\item[\hskip \labelsep {\bfseries #1}]}{\end{trivlist}}

\usepackage{verbatim}
\usepackage[usenames]{color}
\usepackage{hyperref}

\theoremstyle{definition}

\theoremstyle{remark}

\newcommand{\scr}[1]{\ensuremath{\mathcal {#1}}}

\renewcommand{\phi}{\varphi}

\newcommand{\notarrow}{\kern .42em\not\kern -.42em\longrightarrow}

\begin{document}

\centerline {\Large Generalizing Hartogs' Trichotomy Theorem}

\vspace{.1in} \centerline{David Feldman}\centerline{University of
New Hampshire }

\vspace{.1in}\centerline{Mehmet Orhon}\centerline{University of
New Hampshire}

\vspace{.1in}
\centerline{{\bf with an appendix by}}

\vspace{.1in}\centerline{Andreas Blass}\centerline{University of Michigan}

\begin{abstract}
A celebrated argument of F.~Hartogs (1915) deduces the Axiom of
Choice from the hypothesis of comparability for any pair of
cardinals.	We show how each of a sequence of seemingly much
weaker hypotheses suffices. Fixing a finite number $k>1$, the Axiom
of Choice follows if merely any family of $k$ cardinals contains
at least one comparable pair.
\end{abstract}

\vspace{.1in} {\bf Introduction}

\vspace{.1in} The Trichotomy Principle says that for any pair of
sets $A$ and $B$, either a bijection connects the two, or else
precisely one set of the pair will inject into the other. Hartogs
established the logical equivalence, over {\bf ZF}, between the
Trichotomy Principle and the Well-Ordering Principle.  As {\bf ZF}
suffices to prove the Schr\"oder-Bernstein theorem, the heart of
Trichotomy lies in the existence of at least one injection
connecting arbitrary $A$ and $B$ (in whichever direction).

Hartogs' theorem seems a remarkable achievement, even judged
against the large industry that eventually flowered around the
theme of logical equivalents of the Axiom of Choice. Trichotomy
certainly follows quickly from the Well-Ordering Principle, but
deriving enough structure out of  Trichotomy to well-order an
arbitrary set might seem daunting.

Hartogs had the remarkable idea of associating to any set $A$ a
certain ordinal that we write here as $\omega(A)$.	The definition
of $\omega(A)$, as we shall now recall it, depends on von
Neumann's formulation of ordinals. Von Neumann views ordinals as
{\em hereditarily transitive sets}.	 By definition, each element
of a {\em transitive} set $\kappa$ also occurs as a subset of
$\kappa$,
and {\em hereditarily transitive} means transitive with
only transitive elements. An ordinal, viewed as an ordered set,
contains as elements just some other ordinals ordered by inclusion
(or equivalently, by membership); in fact an ordinal's elements
coincide with its proper initial segments. Given any well-ordered
set, one may produce (by transfinite induction) a unique ordinal
with the same order type.

Hartogs first associates to a set $A$ the set $W(A)$ of {\em all}
well-ordered sets modeled on subsets of $A$ and then defines the
ordinal $\omega(A)$ as the set of all ordinals order-isomorphic to
elements of $W(A)$.	 Hartogs argues that $\omega(A)$ cannot inject
into $A$ (lest $\omega(A)$ contain itself). Then Trichotomy
guarantees that $A$ injects into $\omega(A)$.

Here we strengthen Hartogs' theorem, by deriving the well-ordering
principle from seemingly weaker statements.

We say a family of sets ${\scr F}$ {\em contains an injective} if there exists
at least one injective map $i$ which has, for its  source and target, distinct
sets in ${\scr F}.$

We say the {\em $k$-Trichotomy Principle} holds if every family of
cardinality $k$ contains an injective; so 2-Trichotomy recaptures
classical Trichotomy.

Here we shall prove the

\begin{theorem}
For all finite $k$, the $k$-Trichotomy Principle
implies the Well-Ordering Principle.
\end{theorem}

\vspace{.1in} {\bf Notation and basic terminology}

 For sets $A$ and $B$ we write $A\leq B$ if some map injects
$A$ into $B$.  We write $A\cong B$	if a bijection takes $A$ to $B$.
We write $A<B$ if $A\leq B$ but not $A\cong B$.

We call a set {\em infinite} if every finite ordinal injects into
it.	 We write $\omega=\{0,1,2,\ldots\}$ for the smallest infinite
ordinal. We view ordinals as sets of ordinals. In the body of the
paper we shall always take {\em cardinal} in its narrow sense,
meaning an ordinal larger than any of its elements.	 Nevertheless
in our opening and	closing remarks we use {\em cardinal} in the
broader sense, an equivalence class of sets under the relation of
admitting a mutual bijection.

We write $A+B$ for the disjoint union of $A$ and $B$ and $nA$ for the
disjoint union of $n$ copies of $A$.

By a subquotient of a set $A$, we mean any quotient set of a subset of $A$.

We write $2^A$ for the powerset of $A$.

\vspace{.1in} {\bf Proof of the Theorem}

In the sequel, a ($k$) indicates a proposition whose validity depends on
$k$-Trichotomy.

\begin{definition}
 Given any set $A$, well-orderable or not, we
shall write $\omega(A)$ for the smallest ordinal that does not
inject into $A$ (as per the Introduction). $\omega(A)$ always
exists; it contains precisely the ordinals of all possible
well-orderings of subsets of $A$. For well-orderable $A$,
$A<\omega(A)$;	for non-well-orderable $A$, no injection will
connect $A$ and $\omega(A)$ in either direction.
\end{definition}

By its definition, $\omega(A)$ gives a strict upper bound on the
cardinality of a well-ordered set $\kappa$ injecting {\em into} $A$.
But observe also that an injection {\em from} $A$ into $B+\kappa$ hits
fewer than $\omega(A)$ elements of $\kappa$.

\begin{klemma} Given a	cardinal $\kappa$ and sets $A_1 \leq
\cdots \leq A_i \leq \cdots\leq A_k$ with $\omega(A_i)=\kappa$ for all
$i$, there exist $n<m$ and a finite set $R$ such that
$A_m \leq A_n + R$; in case $\kappa>\omega$, we even take
$R=\emptyset$ so that $A_m\cong A_n$.
\end{klemma}

\begin{proof}
Fix any decreasing sequence of infinite	 cardinals
$\kappa_1>\cdots>\kappa_{k}=\kappa$.  (Work backwards, say,
setting each $\kappa_j=\omega(\kappa_{j+1})$.)

Apply $k$-Trichotomy to the family
$A_1+\kappa_1,A_2+\kappa_2,\ldots,A_k+\kappa_k$ to get an
injection $A_m+\kappa_m\rightarrow A_n+\kappa_n$ with $m$ and $n$
distinct.

Focus just now on the restricted injection $i:\kappa_m\rightarrow
A_n+\kappa_n$. As  $i^{-1}(A_n)$ and $i^{-1}(\kappa_n)$ partition
the cardinal $\kappa_m$, $\kappa_m$
admits a bijection with one of
these. But if $i^{-1}(A_n)$ bijects with $\kappa_m$, $\kappa_m$
injects into $A_n$ contradicting $\omega(A_n)=\kappa\leq\kappa_m$.
So $i^{-1}(\kappa_n)$ bijects with $\kappa_m$; $\kappa_m$ injects
into $\kappa_n$; and we must have $n < m$.

Restrict $A_m+\kappa_m\rightarrow A_n+\kappa_n$ next to
$j:A_m\rightarrow A_n+\kappa_n$. $\omega(A_m)=\kappa$ implies
 $j^{-1}(\kappa_n)<\kappa$. Write $R:=j(j^{-1}(\kappa_n))$; so $R\cong
j^{-1}(\kappa_n)<\kappa$. Then $R$ also injects into $A_n$ since
$\omega(A_n)=\kappa$ too.

$\kappa=\omega$ makes $R$ finite. $\kappa > \omega$ makes $\omega$
(and also $R$) inject into $A_n$. So $A_n\cong S+R$, say, and also
$A_n\cong T+\omega$, say.  Now on $R$ infinite (and well-ordered)
we have $R+R\cong R$, so $A_n \cong S+R\cong S+R+R \cong A_n +R$.
On $R$ finite we have $A_n \cong T+\omega\cong T+\omega+R \cong
A_n +R$. Either way,
  $A_n + R
\cong A_n$.	  Thus $A_m$ injects into $A_n$.  But $A_n\leq A_m$ by hypothesis,
so Schr\"oder-Bernstein finally yields $A_n\cong A_m$.
\end{proof}

\begin{lemma}  For an infinite set $A$, the following are equivalent:

(i)	  There exists no injective map from $\omega$ to $A$;\\
(ii)  $A$ bijects with no proper subset of $A$;\\
(iii)  $\omega(A) = \omega$.
\end{lemma}

As usual, we call any such set $A$ {\em infinite Dedekind finite}.

\begin{proof} From (iii), which says all finite ordinals inject into $A$, but
$\omega$ doesn't, we get (i) (plus the infinitude of $A$); then (i) and the
infinitude of $A$ amounts to (iii).

Given an injection $w:\omega \rightarrow A$, form a non-bijective
injection $u:A\rightarrow A$ by setting $u(w(k))=w(k+1)$ for $k\in
\omega$ and having $u$ act as the identity on $A\setminus
w(\omega)$. Clearly $A\cong u(A)=A\setminus\{w(0)\}$, a proper
subset. So if (i) fails, (ii) does too.

Given a non-bijective injection $u:A\rightarrow A$ pick $a\not\in u(A)$. We
must have $a, u(a), u(u(a)), u(u(u(a))),\ldots,u^k(a),\ldots$ all distinct, so
map $\omega$ to $A$ by sending $k$ to $u^k(a)$. So (i) fails if (ii) does.
\end{proof}

Dedekind finiteness for $A$ implies Dedekind finiteness for all $nA$, $n>0$.
(Whenever $\omega$ injects into $nA$, at least one copy of $A$ has an infinite
preimage.)

\begin{klemma}
Infinite Dedekind finite sets do not exist.
\end{klemma}

{\bf Proof} For $A$ infinite Dedekind finite, $\omega(A)=\omega$. Apply Lemma~1
to $A \leq 2A\leq \cdots \leq kA$ for an injection $mA\rightarrow nA+R$ with
$m>n$ and $R$ finite; then $A$ infinite means $R$ injects into $A$ and $nA+R$
injects properly into $mA$. Composing two injections, $mA$ injects properly
into itself. Dedekind finiteness now fails for $mA$ (Lemma~2(ii)), a
contradiction.

\begin{klemma}
For every infinite set $A$ there exists $n$ such that
$nA+nA\cong nA$.
\end{klemma}

\begin{proof}
By Lemma~3 and Lemma~2(iii), $\omega(A)>\omega$.  So apply Lemma~1
to $\{iA\}_{i=1,\ldots,k}$ to get $mA \cong nA$ with $m>n$.

Then $nA\cong mA\cong nA+(m-n)A\cong mA+(m-n)A\cong (2m-n)A$, and continuing
this way, $nA\cong (m+q(m-n))A$ for all $q$.
   Taking $q$ large, we may suppose
we had $m>2n$ in the first place. But $mA\cong nA$ certainly implies
$(m+r)A\cong (n+r)A$, so we may even suppose we had $m=2n$.
\end{proof}

\begin{remark}
In {\bf ZF}, as Lindenbaum and Tarski show, $nP\cong nQ$ implies $P\cong
Q$ (see Conway and Doyle).
\end{remark}

\begin{lemma}
If a set $X$ admits a well-ordering, any subquotient $Y$ of $X$
does too.
\end{lemma}

\begin{proof}
Fix a well-ordering of $X$. By definition, $Y$ bijects with a family
$\{X_y\}_{y\in Y}$ of disjoint subsets of $X$.	Each $X_y$ has a minimal
element $m_y$ according to the order on $X$.  By the disjointness of the $X_y$,
$Y$ bijects with $\{m_y\}$, but $\{m_y\}$ has a well-ordering as a subset of
$X$.
\end{proof}

\begin{lemma}
A set $A$ admits a well-ordering if $A\cong A+A$ and there
exists an injection	 $k:2^A\rightarrow A + \kappa$ for some ordinal $\kappa$.
\end{lemma}

\begin{proof}
We aim to exhibit $A$ as a subquotient of $\kappa$ (Lemma~5).

Write $A=B+C$, with $A\cong B \cong C$; exhibiting $C$ as a subquotient of
$\kappa$ suffices.

For $c\in C$, define $S_c:=\{X \subseteq A| X\cap C=\{c\}\}$. Naturally
$S_c\cong 2^B\cong 2^A$. Certainly $S_c\cap S_d=\emptyset$ unless $c=d$. Also
$T_c:=k(S_c)\cap \kappa \not=\emptyset$ since $k(S_c)\cong 2^A>A$ (Cantor).
Therefore $\{T_c\}_{c\in C}$ forms a collection of pairwise disjoint nonempty
sets in $\kappa$.
\end{proof}

{\bf Proof of the Main Theorem}:  Fix any set $S$.	Set $A=nS$ with $n$ so large
that $A+A\cong A$ (Lemma~4).

Write ${\scr P}^{0}(A)=A$, ${\scr P}^{i+1}(A)=2^{{\scr P}^{i}(A)}$ and
$\omega^{1}(A)=\omega(A)$, $\omega^{i+1}(A)=\omega(\omega^{i}(A))$.

$k$-Trichotomy for the family $\{{\scr P}^i(A)+\omega^{k-i}({\scr
P}^{k-1}(A))\}_{i=0,\ldots,k-1}$ yields an injection
$${\scr P}^i(A)+\omega^{k-i}({\scr P}^{k-1}(A))\rightarrow
{\scr P}^j(A)+\omega^{k-j}({\scr P}^{k-1}(A))$$ which in turn restricts to an
injection
$$\omega^{k-i}({\scr P}^{k-1}(A))\rightarrow {\scr P}^j(A)+\omega^{k-j}({\scr P}^{k-1}(A))$$
that possesses a (well-ordered) cardinal for its domain.
As such we obtain the existence of still another injection, either of the form
$$\omega^{k-i}({\scr P}^{k-1}(A))\rightarrow {\scr P}^j(A)$$
or
$$\omega^{k-i}({\scr P}^{k-1}(A))\rightarrow \omega^{k-j}({\scr P}^{k-1}(A)).$$
Whereas Hartogs' original argument actually rules out an injection
of former sort (since $j \leq k-1$), the existence of an injection of
the latter sort now forces $i>j$.

The original injection also restricts to
$${\scr P}^i(A)\rightarrow{\scr P}^j(A)+\omega^{k-j}({\scr P}^{k-1}(A))\ .$$
so we get an injection
$${\scr P}^i(A)\rightarrow{\scr P}^{i-1}(A)+\omega^{k-j}({\scr P}^{k-1}(A))\ .$$

Thus ${\scr P}^{i-1}(A)$ satisfies the hypothesis of Lemma~6, so admits a
well-ordering. But $A<2^A$ (by singletons) and then $A<{\scr P}^{i-1}(A)$ (by
induction), so a well-ordering of ${\scr P}^{i-1}(A)$ induces a well-ordering
of $A$ which induces a well-ordering of $S$.

\begin{remark}
One may read our main theorem to say that the existence of a pair of
incomparable cardinals entails the existence of a family of $k$ mutually
incomparable cardinals.	 Amidst two incomparable cardinals, at least one
cardinal  admits no well-ordering.	Blass' appendix below offers a direct
construction of a family of $k$ incomparable cardinals starting from any single
set that admits no well-ordering.
\end{remark}

\begin{remark}	We can find no obvious modification of our
methods to show that $\omega$-trichotomy (meaning that any
countable family of cardinals contains at least one comparable
pair) implies AC.  While we hope to return to this matter, for now
we merely record a few observations.

First, one must distinguish between $\omega$-trichotomy	 and
$\infty$-trichotomy (meaning that any infinite family of cardinals
contains at least one comparable pair).	 Certainly
$\infty$-trichotomy immediately implies $\omega$-trichotomy, but
the converse remains unclear on account of the possibility of an
infinite Dedekind finite family of incomparable cardinals.

Combinatorially (if not set-theoretically), $\omega$-trichotomy
actually turns out stronger than it sounds.	 Recall that a special
case of Ramsey's theorem guarantees that any two-colored complete
graph on a countably infinite set of vertices possesses a
monochromatic subgraph.	 Now, given a countably infinite family of
cardinals, we may regard these as the vertices of a complete graph,
taking {\em comparability} and {\em non-comparability} as the colors.
Ramsey's theorem then promises either an infinite subfamily of
mutually comparable cardinals or an infinite subfamily of mutually
incomparable cardinals; as $\omega$-trichotomy rules out the latter,
$\omega$-trichotomy guarantees that every countably infinite family
of cardinals contains an infinite chain.

We get even more on account of the effective nature of certain proofs
of Ramsey's theorem.  Recall how the proof can run.	 Call the two
colors $c_1$ and $c_2$ and write $c(v_i,v_j)$ for the color of the
edge connecting vertices $v_i$ and $v_j$.  Filter the vertex set
successively in order to produce an infinite subgraph in which
$c=c(v_i,v_j)$ depends only on $i$ at least for $j>i$ and assign
vertex $v_i$ color $c$. In the end, we will have assigned at least
one of the two colors to infinitely many of the vertices that
survive. If we have used {\em exactly} one of the colors infinitely
often, we pass to the subset of vertices that received that color;
otherwise we pass to the subset of vertices that received color
$c_1$.	The proof thus explicitly {\em canonizes} a particular
monochromatic subgraph of our graph.

Combining the results of the previous two paragraphs, from a given
countably infinite family of cardinals,	 if we have
$\omega$-trichotomy, we can isolate within the family a canonical
infinite chain.	 Now we can remove that chain and find another.
Indeed, by induction, $\omega$-trichotomy implies that we may
partition any countably family of distinct cardinals into infinite
chains (and perhaps a finite residual).	 (Observe that the induction
does not depend upon AC because we remove a well-determined chain at
each step.)	  If we like, we can continue applying
$\omega$-trichotomy to cross-sections of such a partition to adduce
still more structure.

Finally, we describe a very simple topos example as weak evidence
against the provability of AC from $\omega$-trichotomy. Consider the
topos of sheaves of ({\bf ZFC}) sets over a two point discrete space.
A (global) injection in this topos means a set injection at each
stalk.	View the elements of this topos simply as ordered pairs of
sets, so one has $(\kappa_1,\kappa_2)$ incomparable with
$(\lambda_1,\lambda_2)$ if $\kappa_1 > \lambda_1$ and $\kappa_2 <
\lambda_2$ or vice-versa. Certainly finite anti-chains of any length
exist. An infinite anti-chain, however, would entail an infinite {\em
descending} chain of cardinals in one coordinate corresponding to any
infinite {\em ascending} chain of cardinal in the other, so
impossible.	 Simplicity notwithstanding, this example shows that no
sufficiently constructive method can extrapolate, from
counterexamples to $k$-trichotomy for all finite $k$, to a
counterexample to $\omega$-trichotomy.	Perhaps a suitable
construction, first with ur-elements and then with forcing, can
exploit this same device in the context of classical models of set
theory.
\end{remark}

\centerline{{\bf Appendix by Andreas Blass}}

\vspace{.1in}
\begin{appendix}
For any set $X$, define $\scr Q(X):=X \times \scr P(X)$ and
$\scr Q^{j}(X):=\scr Q(\scr Q^{j-1}(X))$.
Because of the injection
$X\to\scr Q(X)$ sending any $x$ to $(x,\emptyset)$, we have $X\leq\scr
Q(X)$ and therefore, by induction, $X\leq\scr Q^j(X)$ for all $j$.

\begin{lemma}  For a set $X$ and an ordinal $\kappa$, an injection
$\theta:\scr Q(X)\to X+\kappa$ induces a canonical well-ordering of $X$.
\end{lemma}

\begin{proof}
For $x\in X$, Cantor's theorem prevents $\theta(\{x\}\times\scr P(X))$
from being included in $X$.	 Thus, as $x$ varies through $X$, the sets
$T_x=\kappa\cap\theta(\{x\}\times\scr P(X)$ are nonempty, pairwise
disjoint subsets of $\kappa$.  This exhibits $X$ as a subquotient of
$\kappa$, so Lemma~5 provides a well-ordering of $X$.
\end{proof}

For set $X$ and integer $k>1$, define well-ordered cardinals
$$\kappa_0(X,k),\dots,\kappa_{k-1}(X,k)$$ by the recursion
$$
\kappa_0(X,k):=\omega(\scr Q^{k-1}(X))\quad\text{and}\quad
\kappa_{i+1}(X,k)=\omega(\scr Q^{k-i-1}(X)+\kappa_i(X,k)).
$$	For $X$ and $k$ fixed, the sequence of cardinals
$\kappa_i:=\kappa_i(X,k)$ increases strictly.

Claim: Whenever $X$ carries no well-ordering, $\{\scr
Q^{k-i-1}(X)+\kappa_i(X,k)\}$
constitutes an explicit example of $k$ pairwise incomparable sets.

If $\{\scr Q^{k-i-1}(X)+\kappa_i\}$ doesn't violate $k$-Trichotomy,
we must have an injection
of the form $\scr Q^{k-i-1}(X)+\kappa_i \rightarrow \scr
Q^{k-j-1}(X)+\kappa_j$, with $i\neq j$.


In particular, $\kappa_i\leq Q^{k-j-1}(X)+\kappa_j$, and so, by
definition of $\kappa_{j+1}$, we have $\kappa_i<\kappa_{j+1}$.	From
this and $i\neq j$, we get that $i<j$.

Now writing $Y:=\scr Q^{k-j-1}(X)$, we have
$$
\scr Q^{j-i}(Y)=\scr Q^{k-i-1}(X)\leq \scr Q^{k-i-1}(X)+\kappa_i  \leq
\scr Q^{k-j-1}(X)+\kappa_j\cong Y+\kappa_j .
$$
Since $j> i$, we obtain an injection $\theta:\scr
Q(Y)\to Y+\kappa_j$.  Lemma~7 now makes $Y$ well-orderable, and $X$ too,
because $X\leq Y$.
\end{appendix}

\bibliographystyle{amsplain}

\end{document}